\documentclass[preprint,12pt]{article}

\usepackage{amssymb}
\usepackage{amsthm}
\usepackage{amsmath}
\usepackage{enumerate}

\usepackage{ytableau}
\usepackage{graphicx}
\usepackage[all]{xy}

\numberwithin{equation}{section}
\newtheorem{thm}[equation]{Theorem}
\newtheorem{lem}[equation]{Lemma}
\newtheorem{cor}[equation]{Corollary}
\newtheorem{prop}[equation]{Proposition}
\newtheorem{alg}[equation]{Algorithm}
\newtheorem{rem}[equation]{Remark}
\newtheorem{defin}[equation]{Definition}
\newtheorem{ex}[equation]{Example}

\newcommand\ov{\overline}

 \newcommand\id{\mbox{{\rm id}}}
 \newcommand\Hom{\mbox{\rm Hom}}
 \newcommand\End{\mbox{\rm End}}
 \newcommand\Aut{\mbox{\rm Aut}}
 \newcommand\Ker{\mbox{\rm Ker}}
 
 \newcommand\Coker{\mbox{\rm Coker}}
 \newcommand\mmod{\mbox{\rm mod}}
\newcommand\Gl{\mbox{\rm Gl}}

\newcommand\natleq{\leq_{\rm nat}}
\newcommand\domleq{\leq_{\rm dom}}

\newcommand\homleq{\leq_{\rm hom}}

\newcommand\degleq{\leq_{\rm deg}}

\newcommand\mrep{\mbox{ $^aV^b$}}

\renewcommand{\thefootnote}{}

\begin{document}






\thispagestyle{empty}
\color{black}
\phantom m\vspace{-2cm}

\bigskip\bigskip
\begin{center}
{\large\bf Applications of Littlewood-Richardson tableaux to computing generic extension of semisimple invariant subspaces of nilpotent linear operators}
\end{center}

\smallskip

\begin{center}
Mariusz Kaniecki and Justyna Kosakowska\footnote{E-mail addresses: kanies@mat.umk.pl (M.Kaniecki) justus@mat.umk.pl (J.Kosakowska)}

\vspace{1cm}

Faculty of Mathematics and Computer Science,\\
    Nicolaus Copernicus University\\
    Chopina 12/18, 87-100 Toru\'n, Poland
\vspace{1cm}    


\medskip \parbox{10cm}{\footnotesize{\bf MSC 2010:}
Primary:
47A15  
Secondary:
05A17, 
16G20,  
16Z05  
}

\medskip \parbox{10cm}{\footnotesize{\bf Key words:}
generic extensions, degenerations, partial orders,
nilpotent operators, invariant subspaces, Littlewood-Richardson
tableaux}

\end{center}

\begin{abstract}
 The main aim of the paper is to present a~combinatorial algorithm
 that, applying Littlewood-Richardson tableaux with entries equal to $1$, computes
 generic extensions of semisimple invariant subspaces of nilpotent linear operators.
 Moreover, we discuss geometric properties of generic extensions and their connections
 with combinatorics.
\end{abstract}






\section{Introduction}

\let\thefootnote\relax\footnote{This research did not receive any specific grant from funding agencies in the public, commercial, or not-for-profit sectors.}
  The paper is motivated by results presented in \cite{kos-sch,kos-sch2}, where there are investigated  relationships
  between Littlewood-Richardson tableaux and  geometric properties of invariant subspaces of nilpotent linear operators.
 It is observed there that these relationships  are deep and interesting.  On the other hand, in \cite{reineke}, the existence of generic extensions for
  Dynkin quivers is proved and their connections  with Hall algebras are investigated.
 Moreover, by results presented in
\cite{bongartz,dengdu,dengdumah,reineke} generic extensions of nilpotent linear operators
exist and the operation of taking the generic
extension provides the set of all isomorphism classes of
nilpotent linear operators with a~monoid structure. There are many results
concerning this monoid and its properties (see \cite{dengdu,dengdumah,hubery,kos,reineke}).

In the paper we investigate generic extensions of semisimple invariant subspaces of nilpotent linear operators. We show how combinatorial properties of Littlewood-Richardson tableaux are connected with
geometric and algebraic properties of semisimple invariant subspaces. In particular, we present a~combinatorial algorithm
 that, applying operations on Littlewood-Richardson tableaux with entries equal to $1$, computes
 generic extensions of semisimple invariant subspaces of nilpotent linear operators.


Let $K$ be an~algebraically closed field. By {\bf a~nilpotent linear operator} we mean  a~pair
$(X_2,\varphi)$, where $X_2$ is a~finite dimensional $K$-linear space and $\varphi:X_2\to X_2$ is a~nilpotent $K$-linear
endomorphism. {\bf An~invariant subspace} of a~nilpotent linear operator $X_2=(X_2,\varphi)$ is a~triple
$(X_1,X_2,f)$, where $X_1$ is a~linear space and $f:X_1\to X_2$ is a~$k$-monomorphism such that $\varphi(f(X_1))\subseteq f(X_1)$. An~invariant subspace
$(X_1,X_2,f)$ is said to be {\bf semisimple}, if $\varphi(f(X_1))=0$.

Let $a$ and $b$ be the $K$-dimensions of $X_1$ and $X_2$, respectively.
Fix linear bases of $X_1$ and $X_2$.
An~invariant subspace $X=(X_1,X_2,f)$ one can interpret as the point $X=(f,\varphi)$ of the affine variety
$\mathbb{M}_{b\times a}(K)\times \mathbb{M}_b(K),$
where $\mathbb{M}_{b\times a}(K)$ is the set of $b\times a$ matrices with coefficients in $k$, $\mathbb{M}_{b}(K)=\mathbb{M}_{b\times b}(K)$ and
$\varphi,f$ are the matrices of $\varphi$ and $f$
in the bases fixed above (we use the same character for a~map and its matrix). We also use the same character $X$ for an~invariant subspace and for its corresponding point in the associated variety.
Let ${}^aV^b\subseteq \mathbb{M}_{b\times a}(K)\times \mathbb{M}_b(K)$
be the subset consisting of all points that correspond to semisimple invariant subspaces of nilpotent linear operators.
The sets ${}^aV^b$ are locally closed subsets of
$\mathbb{M}_{b\times a}(K)\times \mathbb{M}_b(K)$ invariant under the action of the algebraic group $G=\Gl(a,b)=\Gl(a)\times \Gl(b)$, see Section \ref{sec-partial-orders} for details.
Given a~semisimple invariant subspace $X=(X_1,X_2,f)$
denote by $\mathcal{O}_X$
the orbit of $X$ in $\mrep$ under the action of $G$.

Given subsets $\mathcal{Y}\subseteq \mrep$, $\mathcal{X}\subseteq {}^{c}V^d$,
following \cite{reineke}, we define a~subset
$$
\mathcal{E}(\mathcal{Y},\mathcal{X})\subseteq {}^{a+c}V^{b+d}
$$
of all extensions of $Y\in \mathcal{Y}$ by $X\in \mathcal{X}$.

One of the main results of the paper is the following theorem.

 \begin{thm} \label{thm-main1}
{\rm (1)} Given two semisimple invariant subspaces $X,Y$,
  the set $\mathcal{E}(\mathcal{O}_Y,\mathcal{O}_X)$ contains the unique dense orbit $\mathcal{O}_{Z}$. In particular $\overline{\mathcal{O}_{Z}}=
\overline{\mathcal{E}(\mathcal{O}_Y,\mathcal{O}_X)}$, where $\overline{\mathcal{O}_X}$ denotes the Zariski closure of $\mathcal{O}_X$.

{\rm (2)} The object $Z$ given in {\rm (1)} is isomorphic with $Y\ast X$ constructed by 
Algorithm \ref{algorithm}.

 {\rm (3)} $\overline{\mathcal{E}(\mathcal{O}_Y,\mathcal{O}_X)}=\mathcal{E}(\overline{\mathcal{O}_Y},\overline{\mathcal{O}_X})=\overline{\mathcal{O}_{Y* X}}$.

 {\rm (4)} The set $\overline{\mathcal{E}(\mathcal{O}_Y,\mathcal{O}_X)}=\mathcal{E}(\overline{\mathcal{O}_Y},\overline{\mathcal{O}_X})$ is irreducible.
 \end{thm}

 The~invariant subspace $Z=Y* X$ that correspond to the orbit  $\mathcal{O}_{Z}$ given in Theorem
 \ref{thm-main1} (1) is called  the generic extension of $Y$ by $X$.
 In Section \ref{section-extentions} we  
  define (equivalently) $Y*X$ as an~extension of $Y$ by $X$ with minimal dimension of its endomorphism ring.
 We present the proof in Sections
 \ref{section-extentions} and \ref{section-extentions-geom}.
 The condition (1) of Theorem \ref{thm-main1} is true for algebras of finite representation type, see \cite{bongartz, reineke}, while the conditions (3), (4) in general fail
 for this class of algebras, see \cite[page 150, Remarks a)]{reineke} for a~counterexample. Therefore one can use the results of \cite[Section 6]{bongartz}, \cite[Section 2]{reineke} to prove the statement (1) easily. However, our
 proof of (1) and (2) is combinatorial
 and provides an~algorithm that given
 $X,Y$
 applying operations on Littlewood-Richardson tableaux computes $X* Y$.

 In Section \ref{section-dim-orb} we give a~formula for dimension of an orbit $\mathcal{O}_X$, for given semisimple invariant subspace $X$. More precisely, we prove the following.

 \begin{thm}\label{thm-second-main}
 Let $X\in \mrep$ and let $(\alpha,\beta,\gamma)$ be the triple of partitions associated with $X$, see Section {\rm \ref{section-dim-orb}} for details. We have
  $$\dim \mathcal{O}_X\;=\;b^2+a^2-n(\alpha)-n(\beta)-n(\gamma)-b,$$
  where for a~partition $\lambda$ we set $n(\lambda)=\sum_{i}\lambda_i(i-1)$.
 \end{thm}

 The paper is organized as follows.
 \begin{itemize}
  \item In Section \ref{sec-intro} we recall basic definitions, notation and facts concerning invariant subspaces.

  \item In Section \ref{section-dim-orb}
  we prove Theorem \ref{thm-second-main}.
  
  \item Section \ref{sec-partial-orders} contains description of three partial orders (i.e. $\domleq$, $\homleq$ and $\degleq$) defined on the set
  of isomorphism classes of semisimple invariant subspaces.
  In Theorem \ref{thm-main} we prove that these orders are equivalent. This result is  used in the proof of Theorem \ref{thm-main1}.
  
  \item In Sections \ref{section-extentions} and \ref{sub-gen} we present a~combinatorial algorithm that computes generic extensions, illustrate it by examples and prove that
  the operation of taking the generic extension is associative.
  We also describe generators of the associative monoid of generic extensions.
  \item In Section \ref{section-extentions-geom} we discuss geometric properties of the set $\mathcal{E}(\mathcal{Y},\mathcal{X})$ and we prove Theorem \ref{thm-main1}.
  
 \end{itemize}

\section{Notation and definitions}
\label{sec-intro}

\subsection{Invariant subspaces}
Let $K$ be an~arbitrary field
and let $K[T]$ be the $K$-algebra of polynomials with one variable.
It is well-known that a~nilpotent linear operator $(X_2,\varphi)$
can be identified with the nilpotent $K[T]$-module $X_2$ via $T\cdot v=\varphi(v)$ for all $v\in X_2$,
see \cite[Chapter I, Example 2.6]{ASS1}. In the paper we will use this identification.

Let $\alpha=(\alpha_1\geq\ldots\geq\alpha_n)$ be a~partition.  We identify a~partition with
the corresponding Young diagram (parts of the partition corresponds to columns of its Young diagram).  Denote by
$|\alpha|=\alpha_1+\ldots +\alpha_n$ the length of $\alpha$
and by $\overline{\alpha}$ the partition conjugated to
$\alpha$, i.e. $\overline{\alpha}$ is given by the Young diagram that is the transposition of the Young diagram of $\alpha$.  Given partitions $\alpha$, $\beta$ we denote
by $\alpha\cup \beta$ the union of these partitions,
i.e. the multiset of parts of  $\alpha\cup \beta$ is the union of multisets of parts of $\alpha$ and $\beta$.

For a partition $\alpha=(\alpha_1\geq\ldots\geq\alpha_n)$
we denote by $N_\alpha=N_\alpha(K)$ the nilpotent linear operator of type $\alpha$,
i.e. the finite dimensional $K[T]$-module
$$N_\alpha=N_\alpha(K)=K[T]/(T^{\alpha_1})\oplus\ldots\oplus K[T]/(T^{\alpha_n}).$$
Note that the function $\alpha\mapsto N_\alpha$
defines a~bijection between the set of all partitions and the set of all isomorphism classes of nilpotent linear operators.
Since $T^{\alpha_1}N_\alpha=0$, the module $N_\alpha(K)$ has the natural $K[T]/(T^{\alpha_1})$-module
structure.

Consider the $k$-algebra
$$ \Lambda=\left( \begin{array}{cc} K[T] &K[T] \\ 0&K[T] \end{array}\right)$$
and denote by $\mmod(\Lambda)$ the category of all finite dimensional right $\Lambda$-modules,
and by $\mmod_0(\Lambda)$ the full subcategory of $\mmod(\Lambda)$ consisting of all modules for which the
element $T=\genfrac(){0pt}{}{T\;0}{0\;T}$ acts nilpotently.
It is well known that objects of $\mmod_0(\Lambda)$
may be identified with  systems $(N_\alpha,N_\beta,f),$
where $\alpha$, $\beta$ are partitions and $f: N_\alpha\to N_\beta$ is
a~$K[T]$-homomorphism.
Let  $X=(N_\alpha,N_\beta,f)$, $Y=(N_{\alpha'},N_{\beta'},f')$ be objects of
$\mmod_0(\Lambda)$.
A~morphism $\Psi:X\to Y$ is a~pair $(\psi_1,\psi_2)$, where
$\psi_1:N_\alpha\to N_{\alpha'}$, $\psi_2:N_\beta\to N_{\beta'}$ are homomorphisms
of $K[T]$-modules such that $f'\psi_1=\psi_2f$,
see \cite[A.2, Example 2.7]{ASS1}.

Denote by $\mathcal S$ or $\mathcal{S}(K)$ the full subcategory of
$\mmod_0(\Lambda)$ consisting
of all objects $N=(N_\alpha,N_\beta,f)$,
where $f$ is a~monomorphism. Note that objects of $\mathcal{S}$ may
be identified
with invariant subspaces of nilpotent linear operators.

For a natural number $n$,
let $\mathcal S_n=\mathcal S_n(K)$ be the full subcategory of $\mathcal S$
consisting of all systems $(N_\alpha,N_\beta,f)$
such that $\alpha_1\leq n$.

It is easy to see that the objects in $\mathcal S_1$ are
semisimple invariant subspaces.

We will denote by $\mathcal S_a^b$ the full subcategory of $\mathcal S_1$ consisting of all objects $(N_\alpha,N_\beta,f)$ such that $|\alpha|=a$, $|\beta|=b$.
It is easy to see any object of $\mathcal{S}^b_a$
can be identified with a~$\Lambda^b$-module, where
$$ \Lambda^b=\left( \begin{array}{cc} K[T]/(T) &K[T]/(T) \\ 0&K[T]/(T^b) \end{array}\right).$$

\subsection{Pickets}
\label{section-pickets}

The  category $\mathcal{S}_1(K)$ is of particular interest for us
in this paper. It has the discrete representation type
(i.e. for any $d\in\mathbb{N}$ there is only finitely many isomorphism classes of objects of dimension $d$).  Each indecomposable object is
isomorphic to a {\it picket} that is, it has the form
$$P_0^m=(0,N_{(m)},0)$$
or
$$P_1^m=(N_{(1)},N_{(m)},\iota)$$
where $\iota(1)=T^{(m-1)}$, see \cite{bhw}.
Whenever we want to emphasize the dependence on the field $K$, we will write $P_\ell^m=P_\ell^m(K)$. It follows that
for all natural numbers $a,b$ the category $\mathcal{S}_a^b\subseteq\mathcal{S}_1$
has finite representation type, i.e. there exists only finitely many
isomorphism classes of objects in  $\mathcal{S}_a^b$.

Thanks to this classification we can associate with any object $X$ of $\mathcal{S}_1(K)$ the~LR-tableau
$\Gamma(X)$ with entries equal to $1$
(or the Pieri tableau). We prefer the name
Littlewood-Richardson tableaux, because
of their connections with invariant subspaces, see \cite{macd,klein68,kos-sch,kos-sch2}.
Let $(\beta,\gamma)$ be a~pair of partition. An LR-tableau of type $(\beta,\gamma)$ (or Pieri tableau) is a~tableau $(\gamma\subseteq\beta)$ such that $\beta\setminus \gamma$ is a~horizontal strip, i.e. $\gamma_i\leq\beta_i\leq\gamma_i+1$ for all $i$. In other words it is a~skew diagram of shape $\beta\setminus \gamma$ with entries all equal to $1$.
In the following table we list corresponding LR-tableaux for indecomposables.

\begin{center}
\begin{tabular}{|c|c|c|}\hline
\multicolumn3{|c|}
             {\raisebox{-1ex}[0mm]{\bf LR-tableaux
                  for the indecomposable objects of $\mathcal S_1$}}
             \\[2ex] \hline
X & $P_0^m$ & $P_1^m$ \\
\hline
 $\Gamma(X)$ &$
m\left\{\ytableausetup{centertableaux}\ytableausetup{smalltableaux}\begin{ytableau}
\\
\none \\
\none[\vdots]\\
\none\\
\\
\end{ytableau} \right.
$ &
$m\left\{\ytableausetup{centertableaux}\ytableausetup{smalltableaux}
\begin{ytableau}
\\
\none \\
\none[\vdots]\\
\none\\
\\
1
\end{ytableau}\right.$
\\ \hline
\end{tabular}
\end{center}
Let $X,X'$ be objects of $\mathcal{S}_1$ and let
$\beta,\gamma$ and $\beta',\gamma'$ be partitions
determining $\Gamma(X)$ and $\Gamma(X')$, respectively.
The LR-tableau of the~direct sum $X \oplus X'$ is given by $\beta\cup \beta^\prime, \gamma\cup\gamma'$.
Recall that parts of a~partition correspond to columns of its Young diagram.

\begin{ex}
The object  $X=P_0^7\oplus P_1^7\oplus P_1^5\oplus P_1^2\oplus P_1^2\oplus P_0^1$
has the following LR-tableau
$$\Gamma(X)=\ytableausetup{centertableaux}
\ytableaushort
{\none\none\none\none\none\none,\none\none\none11,\none\none\none,\none\none\none,\none\none1,\none\none,\none1}
* {6,5,3,3,3,2,2}
$$
Note that the type of $\Gamma(X)$ is $(\beta,\gamma)$, where $\beta=(7,7,5,2,2,1)$ and $\gamma=(7,6,4,1,1,1)$.
\end{ex}

\begin{rem}
 \begin{enumerate}
  \item The association $X\mapsto \Gamma(X)$
  establishes a~bijection between the set of all isomorphism classes of semisimple invariant subspaces and the set of all Littlewood-Richardson tableaux
  with entries equal to $1$.
  \item The invariant subspace $X$
  is uniquely determined (up to isomorphism) by a~pair of
  partitions $\gamma^X\subseteq \beta^X$, where
  $\beta^X$ is defined by the Young diagram of
  all boxes of $\Gamma(X)$ and
  $\gamma^X$ is given by the Young diagram of all empty boxes of $\Gamma(X)$.
  \item If $X$ is given by $\gamma^X\subseteq \beta^X$,
  then the number of indecomposable direct summands
  of $X$ is equal to $\overline{\beta_1}$.
 \end{enumerate}

\end{rem}

For each pair $(X,Y)$ of indecomposable objects in $\mathcal S_1(K)$ we determine
in the table below
the dimension of the $K$-space $\Hom_{\mathcal S}(X,Y)$ of all $\mathcal{S}$-homomor\-phisms from $X$ to $Y$, see \cite[Lemma 4]{sch} and \cite{kos-sch}.

\begin{center}
\begin{tabular}[h]{|r|@{}c@{}|@{}c@{}|}\hline
\multicolumn{3}{|c|}{\bf Dimensions of spaces $\Hom_{\mathcal{S}}(X,Y)$}\\
\hline
  $\quad X\quad $ &  $Y=P_0^m$ &  $P_1^m$  \\
\hline \hline
$P_0^\ell$ & $\min\{\ell,m\}$ & $\min\{\ell,m\}$ \\
\hline
$P_1^\ell$ & $\min\{\ell -1,m\}$   & $\min\{\ell,m\}$ \\
\hline
\end{tabular}
\end{center}

\section{Dimensions of orbits}
\label{section-dim-orb}

The main aim of this section is to prove Theorem \ref{thm-second-main}.

Assume that $K$ is an algebraically closed field.
For natural numbers $a\leq b$ we consider the affine variety
$$\mathbb{M}_{b\times a}(K)\times \mathbb{M}_b(K),$$
where $\mathbb{M}_{b\times a}(K)$ is the set of $b\times a$ matrices with coefficients in $k$
and $\mathbb{M}_{b}(K)=\mathbb{M}_{b\times b}(K)$.
We work with the Zariski topology and with the  induced topology for all subsets of $\mathbb{M}_{b\times a}(K)\times \mathbb{M}_b(K)$.
We define ${}^aV^b={}^aV^b(K)$ as a~subset of  $\mathbb{M}_{b\times a}(K)\times \mathbb{M}_b(K)$ consisting of all points $X=(f,\varphi)$,
such that $\varphi^b=0,\, \varphi\cdot f=0$ and
$f$ has maximal rank.
On  $\mrep$  the algebraic group $G=\Gl(a,b)=\Gl(a)\times \Gl(b)$ acts via $(g,h)\cdot (f,\varphi)=(hfg^{-1},
 h\varphi h^{-1})$,
 where $\Gl(a)$ is the general linear group of all invertible $a\times a$ matrices.
For a~point $X\in \mrep$, denote by $\mathcal{O}_X$ the orbit of $X$ under the action of $G$ and by $G.X$ the stabilizer of $X$ in $G$.
\smallskip

\begin{rem}
\begin{enumerate}
 \item If $K$ is an~algebraically closed field, then the set of points of $\mrep$ is in bijection with the set
of objects of $\mathcal S_a^b$. This bijection is given by
 $$ X=(N_\alpha,N_\beta,f) \longrightarrow X=(f,\varphi_\beta), $$
 where $N_\beta=(K^{b},\varphi_\beta)$.
 We will identify an~object $X$ of $\mathcal{S}_a^b$ with the corresponding
 point $X$ of $\mrep$.
 \item Note that the $G$ - orbits in $\mrep$ are in $1 - 1$ - correspondence with the isomorphism classes of objects
 in $\mathcal S_a^b$.
\end{enumerate}
\end{rem}\smallskip

\begin{proof}[Proof of Theorem~\ref{thm-second-main}]
Let $X=(N_\alpha,N_\beta,f)\in \mathcal{S}_a^b$, let
$X=(f,\varphi)\in \mrep$ and let
$\Coker\, f\simeq N_\gamma$. Here
the partition triple associated with $X$ is $(\alpha,\beta,\gamma)$.

It is well known that
\begin{equation}\label{dim-O_f}
\dim \mathcal{O}_X=\dim \Gl(a,b) - \dim \Gl(a,b).X.
\end{equation}
 Moreover
\begin{equation}
  \dim \Aut_{\mathcal{S}}(X)=\dim \Gl(a,b).X,\label{stab-aut}
\end{equation}
where $\Aut_{\mathcal{S}}(X)$
is the group of $\mathcal{S}$-automorphisms of $X$.
Since $\dim \Gl(a,b)=a^2+b^2$, it is enough to prove that $\dim \Aut_{\mathcal{S}}(X)=n(\alpha)+n(\beta)+n(\gamma)+b$. The set $\Aut_{\mathcal{S}}(X)$ is dense and open in $\End_{\mathcal{S}}(X)$, and therefore $\dim \Aut_{\mathcal{S}}(X)=\dim_K\End_{\mathcal{S}}(X)$. Denote
$[X,X]=\dim_K\End_{\mathcal{S}}(X)$. We prove that $[X,X]=n(\alpha)+n(\beta)+n(\gamma)+b$ by the induction on the number $k$ of the indecomposable direct summands of $X$. If $k=1$, then
$X\simeq P_1^n$ or $X\simeq P_0^n$.
In both cases $n(\alpha)=n(\beta)=n(\gamma)=0$, $n=b$
and $[X,X]=n$. We are done.

Assume that $X\simeq \bigoplus_{i=1}^k P_{\varepsilon_i}^{\beta_i}$, where $k\geq 2$, $\varepsilon_i\in\{0,1\}$ for all $i$. Without loss of generality we can assume that $\beta_1\geq\ldots\geq\beta_k$ are numbered in such a~way that
for all $i$:
\begin{equation}\label{eq-po}
\mbox{if } \beta_i=\beta_{i+1}
\mbox{ and } \varepsilon_i\neq \varepsilon_{i+1}, \mbox{ then } \varepsilon_i=0 \mbox{ and } \varepsilon_{i+1}=1.
\end{equation}

 Let $X=X'\oplus P_{\varepsilon_k}^{\beta_k}$, where $X'=\bigoplus_{i=1}^{k-1}P_{\varepsilon_i}^{\beta_i}$ and let
$s_0=\#\{i\in\{1,\ldots,k-1\}\; ;\; \varepsilon_i=0\}$, $s_1=k-1-s_0$. Denote by $(\alpha',\beta',\gamma')$ the triple of partitions associated with $X'$. Note that
$n(\beta)=n(\beta')+(k-1)\beta_k$ and
$b=|\beta|=|\beta'|+\beta_k$.

First assume that $\varepsilon_k=1$. In this case $X=X'\oplus P_1^{\beta_k}$, $n(\gamma)=n(\gamma')+(k-1)(\beta_k-1)$ and $n(\alpha)=n(\alpha')+s_1$. Then (applying induction and data given in tables at the end of Section \ref{section-pickets}) we get
$$
[X,X]=[X',X']+[X,P_1^{\beta_k}]+[P_1^{\beta_k},P_1^{\beta_k}]+[P_1^{\beta_k},X]=$$
$$n(\beta')+n(\gamma')+n(\alpha')+b'+(k-1)\beta_k+
\beta_k+s_0(\beta_k-1)+s_1\beta_k=
$$
$$
n(\beta)+b+n(\gamma')+(k-1)(\beta_k-1)+n(\alpha')+s_1=n(\alpha)+n(\beta)+n(\gamma)+b.
$$

Now, assume that $\varepsilon_k=0$.
In this case $X=X'\oplus P_0^{\beta_k}$, $n(\alpha)=n(\alpha')$ and $n(\gamma)=n(\gamma')+(k-1)\beta_k$, because of (\ref{eq-po}). Note that the property (\ref{eq-po}) implies that there is no direct summand of $X'$ isomorphic with $P_1^{\beta_k}$
and that $[X,P_0^{\beta_k}]=(k-1)\beta_k$. Then (applying induction, data given in tables at the end of Section \ref{section-pickets}
and the property (\ref{eq-po})) we get
$$
[X,X]=[X',X']+[P_0^{\beta_k},X]+[P_0^{\beta_k},P_0^{\beta_k}]+[X,P_0^{\beta_k}]=$$
$$n(\beta')+n(\gamma')+n(\alpha')+b'+(k-1)\beta_k+
\beta_k+(k-1)\beta_k=
$$
$$
n(\beta)+b+n(\gamma')+(k-1)\beta_k+n(\alpha)=n(\alpha)+n(\beta)+n(\gamma)+b.
$$

Finally we get
 $$\dim \mathcal{O}_X\;=\;b^2+a^2-n(\alpha)-n(\beta)-n(\gamma)-b.$$
\end{proof}

\section{Some partial orders in the category $\mathcal{S}_1$}
\label{sec-partial-orders}

Fix integers $a,b$. We define three partial orders $\domleq$, $\homleq$, $\degleq$ on $\mathcal S^a_b$ and prove that they are equivalent. Results of this section are used in Section \ref{section-extentions} and in the proof of Theorem \ref{thm-main1}.

\begin{defin}
Let $X,Y\in\mathcal S^a_b$ be given by partitions $\gamma^X\subseteq \beta^X$ and $\gamma^ Y\subseteq \beta^Y$, respectively.
\begin{itemize}
 \item We say that $X,Y$ are in the {\bf dominance order}, in symbols $X\domleq Y,$ if partitions
$\beta^X,\gamma^X$ and $\beta^Y, \gamma^Y$ are in the natural partial order,
i.e. $\beta^X\leq_{nat} \beta^Y$ and $\gamma^X\leq_{nat} \gamma^Y$
(we say that partitions $\lambda$ and $\mu$ of the same length are in the natural order $\lambda\leq_{nat}\mu$ if for any $k$ there is $\sum\limits_{i=1}^k\lambda_i\geq\sum\limits_{i=1}^k\mu_i$ or equivalently
$\sum\limits_{i=1}^k\overline{\lambda}_i\leq\sum\limits_{i=1}^k\overline{\mu}_i$, where $\overline{\lambda}$ denotes
the conjugate of $\lambda$).
\item We say that $X,Y$ are in the {\bf hom order}, in symbols $X\homleq Y,$
if $$[Z,X]\leq [Z,Y]$$ for any object $Z$ in $\mathcal{S}(K)$.
Here we write $[X,Y]=\dim_K\Hom_{\mathcal{S}}(X,Y)$ for $\mathcal{S}$-objects $X,Y$.
\end{itemize}

\end{defin}


\begin{prop}\label{prop_homdom}
Let $K$ be an~arbitrary field and let
$X,Y$ be  objects of $\mathcal S_a^b(K)$.
The following conditions are equivalent:
\begin{enumerate}[\rm (1)]
\item $X\domleq Y$,
\item $[Z,X]\leq [Z,Y]$ for any object $Z$ in $\mathcal{S}_1(K)$,
\item $X\homleq Y$,
\item $[Z,X]\leq [Z,Y]$ for any $\Lambda^{b+1}$-module $Z$,
\item $[X,Z]\leq [Y,Z]$ for any $\Lambda^{b+1}$-module $Z$,
\item $[Z,X]\leq [Z,Y]$ for any $\Lambda$-module $Z$.
\end{enumerate}
\end{prop}

\begin{proof}

The equivalence of (3) and (6) is proved in \cite[Lemma 3.3 (2)]{kos-sch}. Since assumptions of \cite[Lemma 3.3 (2)]{kos-sch} do not fit exactly to our situation, we rewrite this proof.

  Any object  $X=(X_1,X_2,f)$ of $\mathcal{S}$ is a~$\Lambda$-module. Thus (6) implies (3). 
  Now we prove
 that  (3) implies (6). Assume that $[Z',X]\leq [Z',Y]$ 
for any object $Z'\in\mathcal{S}$, we show that $[Z,X]\leq [Z,Y]$ 
for any $Z\in\mod(\Lambda)$.
 Let $Z=(Z_1,Z_2,h)\in \mod(\Lambda)$. Note that we can assume that $Z\in \mod_0(\Lambda)$, because otherwise
 $[Z,X]=[Z,Y]=0$.
 Write $X=(X_1,X_2,f)$ where 
$f$ is a~monomorphism. Consider
the object $Z'=(Z_1',Z_2,h')\in \mathcal{S}$,
where $Z_1'=Z_1/\Ker\, h$ and $h':Z_1'\to Z_2$ is induced by $h$. Let 
$a=(a_1,a_2):Z\to X$ be a~morphism. Note that if $x\in \Ker\, h$, then $x\in \Ker\, a_1$
because $f$ is a~monomorphism. Therefore there exists $a_1':Z_1'\to X_1$
such that $a_1=a_1'\circ {\rm can}_1$, where ${\rm can}_1: Z_1\to Z_1'$ is the canonical epimorphism.
 It is easy to see that the pair $a'=(a_1',a_2)$ defines a~morphism $Z'\to Z$.
Writing ${\rm can} = ({\rm can}_1,1): Z\to Z'$, this morphism satisfies $a=a'\circ{\rm can}$.
Thus, ${\rm can}: Z\to Z'$ is a left approximation for $Z$
in $\mathcal S$.  Since ${\rm can}$ is onto, it follows that $[Z,X]=[Z',Y]$.

 Obviously, the condition (6) implies
 the conditions (2) and (4).

 By \cite[page 648]{bongartz0}, (4) and (5) are equivalent and by
 \cite[Proposition 4.2]{kos-sch2},  (3) implies (1).

The object $P_1^0=(N_{(1)},0,0)$ is the only indecomposable $\Lambda^{b+1}$-module that is not an object of $\mathcal{S}_1(K)$. Moreover
$[X,P_1^0]=[Y,P_1^0]=a$.
 Therefore, by \cite[Proposition 4.2]{kos-sch2}
 the condition (1) implies (5).

 We prove that (4) implies (2). Assume that $[Z,X]\leq [Z,Y]$
for any $\Lambda^{b+1}$-module $Z$. We prove that
 $[Z,X]\leq [Z,Y]$
for any object $Z$ in $\mathcal{S}_1$. It is enough to
prove this for indecomposable objects $Z=P_0^m$ and $Z=P_1^m$.
Applying formulae for dimensions of homomorphisms spaces presented in Section \ref{section-pickets} it is easy to see
that for $i=1,0$ and $m>b$ we have $[P_i^m,X]=[P_i^{b+1},X]$.
This proves our claim, because $P_i^{b+1}$ is a~$\Lambda^{b+1}$-module for $i=0,1$.

 Next we show (2) implies (3). Assume that $[Z'',X]\leq [Z'',Y]$
for any object $Z''$ in $\mathcal{S}_1$. We prove that $[Z',X]\leq
[Z',Y]$ for any object $Z'$ in $\mathcal{S}$.
Let $Z'=(Z'_1,Z'_2,h')\in \mathcal S$
  and $Y=(Y_1,Y_2,f)\in \mathcal{S}_1$, so $T\cdot Y_1=0$.
Consider the object $Z''=(Z_1'',Z_2'',h'')\in \mathcal{S}_1$,
where $Z_1''=Z'_1/(T\cdot Z'_1)$, $Z_2''=Z_2'/h'(T\cdot Z'_1)$  and $h'':Z_1''\to Z_2''$ is
the map induced by $h'$.
Since $h'$ is a monomorphism, so is $h''$.  Let ${\rm can}=({\rm can}_1,{\rm can}_2):Z'\to Z''$
be the canonical map.  We show that a morphism
$a=(a_1,a_2):Z'\to Y$ factors over ${\rm can}$.
Since $T\cdot Y_1=0$, we have $T\cdot Z'_1\subseteq \Ker\, a_1$,
so $a_1$ factors over ${\rm can}_1$. We write $a_1=a_1''\circ{\rm can}_1$.
Since $h'(T\cdot Z'_1)\subseteq \Ker\, a_2$, the map $a_2$ factors over ${\rm can}_2: Z'_2\to Z''_2$, i.e.
there is $a_2''$ with $a_2=a_2''\circ{\rm can}_2$.
Since ${\rm can}_1$ is onto, the pair $a''=(a_1'',a_2'')$
is a morphism from $Z''$ to $Y$ in $\mathcal S$ satisfying $a=a''\circ{\rm can}$.

We have seen that ${\rm can}: Z'\to Z''$ is a left approximation for $Z'$ in $\mathcal S_1$.
Since ${\rm can}$ is onto, it follows that $[Z',Y]=[Z'',Y]$.

We are done.
\end{proof}


\begin{defin}
Let $X=(N_\alpha,N_\beta,f)$ and $Y=(N_{\widetilde \alpha},N_{\widetilde \beta},g)$ be objects
in $\mathcal S_a^b(K)$.
The relation $X \degleq Y$
holds if $\mathcal{O}_Y \subseteq \overline{\mathcal{O}_X}$ in $\mrep(K)$,
where $\overline{\mathcal{O}_X}$ is the closure of $\mathcal{O}_X$.
\end{defin}

\begin{thm}\label{thm-main} Let $K$ be an~algebraically closed field and assume that
$X,Y\in \mathcal{S}^b_a(K)$.
The following conditions are equivalent
\begin{enumerate}
 \item $X\domleq Y$,
 \item $X\degleq Y$,
 \item $X\homleq Y$.
\end{enumerate}
\end{thm}

\begin{proof} Let $X,Y\in \mathcal{S}^b_a(K)$. Then $X,Y$ are $\Lambda^b$-modules.

If $K$ is an~algebraically closed field, then by \cite[Corollary, page 1314]{zwara}
we have
$$ X\degleq Y\;\Longleftrightarrow\; X\homleq Y$$
because the category $\mathcal{S}^b_a(K)$ (and the algebra $\Lambda^b$)
is of finite representation type.
Since by Proposition \ref{prop_homdom},
$$ X\domleq Y\;\Longleftrightarrow\; X\homleq Y$$
we are done.
\end{proof}

\section{Generic extensions - an algorithmic approach}
\label{section-extentions}

Let $K$ be an arbitrary field and let $X, Y \in \mathcal S_1(K).$ An object $Z\in \mathcal S_1$ is an extension of $Y$ by $X$ if there exists a~short exact
sequence of the form:
$$
0 \rightarrow  X \rightarrow Z \rightarrow Y \rightarrow 0.
$$

Note that the subcategory $\mathcal{S}_1\subseteq \mathcal{S}$ is not closed under extensions. In this paper we are interested
only in the extensions  that are objects of $\mathcal{S}_1$.

\subsection{Generic extensions - remarks}

If $K$ is an algebraically closed field, then one can define generic extensions in purely
geometric way using Theorem \ref{thm-main1} (1), see page 3. However, in view of the formula
$\dim \mathcal{O}_X=\dim G - \dim_K \End_{\mathcal{S}}(X)$, we present the following equivalent definition.

\begin{defin} The~unique (up to isomorphism) extension $Z$ of $Y$ by $X$ with minimal $K$-dimension of the endomorphism ring $\End_{\mathcal{S}}(Z)$ is called
 {\bf the generic extension} of $Y$ by $X$
 and it is denoted by $Y\ast X$.
\end{defin}

Generic extensions do not exist in general. The existence of generic extensions
in the category $\mathcal{S}_1$ one can prove applying geometric arguments,
see remarks on page 3 below Theorem \ref{thm-main1}.
In this section we 
 present a~combinatorial algorithm that given objects 
 $X,Y\in \mathcal{S}_1$ computes the generic extension $Y\ast X$. 
 This 
 in fact give an independent combinatorial proof of Theorem \ref{thm-main1} (1) and (2).
 
To prove the correctness of Algorithm \ref{algorithm} we have to show that
$Y\ast X$ computed by this algorithm is an~extension
of $Y$ by $X$ (which is done in Lemma \ref{lem_alg_ext}) and that $Y\ast X$
is the unique extension of $Y$ by $X$ with
minimal dimension of the endomorphism
ring (which is done in Corollary
\ref{cor_gen}). The proof of Corollary
\ref{cor_gen} uses Lemma \ref{lem_dommax},
where we show that the extension of $Y$ by $X$ computed by
Algorithm \ref{algorithm} is a~minimal extension in 
the order $\domleq$ (and hence in $\degleq$, by Theorem \ref{thm-main}).

\subsection{Algorithm}

The following algorithm computes the generic extension in the category $\mathcal S_1$ (the proof is presented in next subsections).
For an object $X\in \mathcal S_1$, by $\gamma^X\subseteq \beta^X$ we denote the pair
of partitions uniquely (up to isomorphism) defining $X$. We identify $X$ with the pair $(\gamma^X, \beta^X)$.

\begin{alg}\label{algorithm}

{\rm
{\bf Input.}  $X,Y\in \mathcal S_1$.

{\bf Output.} The generic extension $Z=Y*X\in \mathcal{S}_1$ of $Y$ by $X$.

\begin{enumerate}

  \item $\gamma^Z=\gamma^X+\gamma^Y$

 \item  set $n=0$

\item for any $i=1,\ldots,\min \{\overline{\beta^X_1},\overline{\beta^Y_1}\},$ do
  \begin{enumerate}
   \item put $\beta^Z_i=\beta^X_i+  \gamma^Y_i$
   \item if $\beta^Y_i\neq \gamma^Y_i$, then put $n=n+1$
  \end{enumerate}

\item if $\overline{\beta_1^X}>\min \{\overline{\beta^X_1},\overline{\beta^Y_1}\}$, then
for $i=\min \{\overline{\beta^X_1},\overline{\beta^Y_1}\}+1,\ldots,\overline{\beta_1^X}$ put
$$\beta^Z_i=\beta^X_i,$$
  else
for $i=\min \{\overline{\beta^X_1},\overline{\beta^Y_1}\}+1,\ldots,\overline{\beta_1^Y}$
we set
$$
\beta_i^Z=\beta_i^Y+\mathbf{1}{\{\gamma^Y_i=\beta^Y_i\,{\rm and}\,n>0\}} $$
$$n=n-\mathbf{1}{\{\gamma^Y_i=\beta^Y_i\,{\rm and}\,n>0\}} $$
where
$$
\mathbf{1}{\{\gamma^Y_i=\beta^Y_i\,{\rm and}\,n>0\}}=\left\{\begin{array}{ll}
               1 &  \mbox{if }\gamma_i^Y=\beta_i^Y\,
{\rm and}\,n>0 \\
               0 &  \mbox{otherwise}
              \end{array}
 \right.
$$

\item We set \begin{equation*}
\beta^Z=\beta^Z\cup\alpha
\end{equation*}
where $\alpha=(1,1,\ldots,1)$ is a partition with $n$ copies of $1.$
\end{enumerate}
}
\end{alg}

We illustrate this algorithm by the following examples.

\begin{ex}
Let $X=P_0^4\oplus P_1^4 \oplus P_1^3$ and $Y=P_1^4 \oplus P_0^3\oplus P_0^2\oplus P_1^2\oplus P_0^1 \oplus P_0^1$.
Note that
$$\Gamma(X)=\ytableausetup{centertableaux}
\ytableaushort
{\none\none\none,\none\none\none,\none\none1,\none1}
* {3,3,3,2}\quad
\Gamma(Y)=\ytableausetup{centertableaux}
\ytableaushort
{\none\none\none\none\none\none,\none\none\none1,\none\none,1}
* {6,4,2,1}$$
and $\beta^X=(4,4,3),\gamma^X=(4,3,2)$, $\beta^Y=(4,3,2,2,1,1),\gamma^Y=(3,3,2,1,1,1).$
The Step 1 of Algorithm \ref{algorithm} gives
$$\gamma^Z=\gamma^X+\gamma^Y=(7,6,4,1,1,1).$$
In our case $\min\{\overline{\beta^X_1},\overline{\beta^Y_1}\}=3$ and therefore,
for $i=1,2,3$, the Step 3 of Algorithm \ref{algorithm} sets $\beta^Z_i=\beta^X_i+\gamma^Y_i$.
We have
$$\beta_1^Z=7,\, \beta_2^Z=7,\, \beta_3^Z=5.$$
Note that $n=1.$

For $i>3=\overline{\beta^X_1}=\min \{\overline{\beta^X_1},\overline{\beta^Y_1}\}$ the Step 4 of the algorithm gives

$$
\beta^Z_i=\beta^Y_i+\mathbf{1}{\{\gamma^Y_i=\beta^Y_i\,{\rm and}\,n>0\}}
\; ,\;\;\; n=n-\mathbf{1}{\{\gamma^Y_i=\beta^Y_i\,{\rm and}\,n>0\}}.
$$

For $i=4$ we have:

$$\beta_4^Z=\beta^Y_4=2\;{\rm and}\;n=1\, ,\;{\rm because}\;\gamma_4^Y\neq\beta_4^Y,$$
and for $i=5$:

$$\beta^Z_5=\beta^Y_5+1=2\;{\rm and}\;n=0\,, \;{\rm because}\,\gamma_5^Y=\beta^Y_5.$$
Now $n=0$, and therefore for $i=6$ we get:
$$\beta^Y_6=1.$$
Finally, we get $\beta^Z=(7,7,5,2,2,1)$ and $\gamma^Z=(7,6,4,1,1,1),$
then $$Z=P_0^7\oplus P_1^7\oplus P_1^5\oplus P_1^2\oplus P_1^2\oplus P_0^1$$
and LR-tableau of $Z$ has the form (by the gray colour we marked part of the diagram corresponding to the object $Y$):
$$\Gamma(Z)=\ytableausetup{centertableaux}
\ytableaushort
{\none\none\none\none\none\none,\none\none\none11,\none\none\none,\none\none\none,\none\none1,\none\none,\none1}
* {6,5,3,3,3,2,2}
	* [*(gray)]{6,5,2}$$
\end{ex}

\begin{ex}
Let $X=P_1^4 \oplus P_0^3\oplus P_0^2\oplus P_1^2\oplus P_0^1 \oplus P_0^1$ and $Y=P_0^4\oplus P_1^4 \oplus P_1^3.$
Then $\beta^X=(4,3,2,2,1,1),\gamma^X=(3,3,2,1,1,1)$ and $\beta^Y=(4,4,3),\gamma^Y=(4,3,2)$ and $\min\{\overline{\beta^X_1},\overline{\beta^Y_1}\}=3$.
The Step 1 of Algorithm \ref{algorithm} gets
$$\gamma^Z=\gamma^X+\gamma^Y=(7,6,4,1,1,1).$$
For $i=1,2,3$, the Step 3 of the Algorithm sets $\beta^Z_i=\beta^X_i+\gamma^Y_i$ so we have
$$
\beta_1^Z=8,\, \beta_2^Z=6,\, \beta_3^Z=4,
$$
moreover  $n=2.$

For $\overline{\beta^X_1}\geq i>3=\overline{\beta^Y_1}=\min \{\overline{\beta^X_1},\overline{\beta^Y_1}\}$ the Step 4 of the Algorithm sets
$$
\beta^Z_i=\beta^X_i.
$$
We have:
$$\beta^Z_4=2,\,\beta^Z_5=1,\,\beta^Z_6=1.$$

Parts of the partition $\beta^Z$ for $i=1,2,\ldots,\max \{\overline{\beta^X_1},\overline{\beta^Y_1}\}$ are constructed,
but $n=2\neq 0$. Therefore the Step 5 of Algorithm \ref{algorithm} gives $\beta^Z=\beta^Z\cup (1,1).$
Finally, we get $\beta^Z=(8,6,4,2,1,1,1,1)$ and $\gamma^Z=(7,6,4,1,1,1),$
so $$Z=P_1^8\oplus P_0^6\oplus P_0^4\oplus P_1^2\oplus P_0^1\oplus P_0^1\oplus P_1^1\oplus P_1^1$$
and LR-tableau of $Z$ has the following form (by the gray colour we marked part of the diagram corresponding to the object $Y$):
$$\Gamma(Z)=\ytableausetup{centertableaux}
\ytableaushort
{\none\none\none\none\none\none11,\none\none\none1,\none\none\none,\none\none\none,\none\none,\none\none,\none,1}
* {8,4,3,3,2,2,1,1}
* [*(gray)]{3,3,2,1}
*[*(gray)]{6+2}$$
\end{ex}

\begin{ex}
Let $X=P_1^2$ and $Y=P_0^1 \oplus P_0^1\oplus P_1^1$.
Note that
$$\Gamma(X)=\ytableausetup{centertableaux}\ytableausetup{smalltableaux}
\ytableaushort
{\none,1}
* {1,1}\quad
\Gamma(Y)=\ytableausetup{centertableaux}
\ytableaushort
{\none\none1}
* {3}$$
and $\beta^X=(2),\gamma^X=(1)$, $\beta^Y=(1,1,1),\gamma^Y=(1,1).$
Note that,
for $Z=X*Y,$ we get $\beta^Z=(2,1,1,1)$ and $\gamma^Z=(2,1),$
$Z=P_0^2\oplus P_0^1\oplus P_1^1\oplus P_1^1$
and LR-tableau of $Z$ has the form:
$$\Gamma(Z)=\ytableausetup{centertableaux}
\ytableaushort
{\none\none11,\none}
* {4,1}$$
For $Z'=Y*X,$ we get $\beta^{Z'}=(3,1,1)$, $\gamma^{Z'}=(2,1),$
$Z=P_1^3\oplus P_0^1\oplus P_1^1$
and LR-tableau of $Z'$ has the form:
$$\Gamma(Z')=\ytableausetup{centertableaux}
\ytableaushort
{\none\none1,\none,1}
* {3,1,1}$$
\end{ex}

\begin{ex}
Let $X=P_1^2$ and $Y=P_1^2\oplus P_0^1 \oplus P_0^1\oplus P_1^1$.
Note that
$$\Gamma(X)=\ytableausetup{centertableaux}
\ytableaushort
{\none,1}
* {1,1}\quad
\Gamma(Y)=\ytableausetup{centertableaux}
\ytableaushort
{\none\none\none1,1}
* {4,1}$$
and $\beta^X=(2),\gamma^X=(1)$, $\beta^Y=(2,1,1,1),\gamma^Y=(1,1,1).$
Note that,
for $Z=X*Y,$ we get $\beta^Z=(3,1,1,1,1)$, $\gamma^Z=(2,1,1),$ $Z=P_1^3\oplus P_0^1\oplus P_0^1 \oplus P_1^1\oplus P_1^1$
and
$$\Gamma(Z)=\ytableausetup{centertableaux}
\ytableaushort
{\none\none\none11,\none,1}
* {5,1,1}$$
For $Z'=Y*X,$ we get $\beta^{Z'}=(3,2,1,1)$, $\gamma^{Z'}=(2,1,1),$
$Z=P_1^3\oplus P_1^2\oplus P_0^1\oplus P_1^1$
and
$$\Gamma(Z')=\ytableausetup{centertableaux}
\ytableaushort
{\none\none\none1,\none1,1}
* {4,2,1}$$
\end{ex}

\subsection{Proof of correctness of Algorithm \ref{algorithm}}

The following lemma shows that Algorithm \ref{algorithm} computes an~extension
of $Y$ by $X$.

\begin{lem}\label{lem_alg_ext}

Let $X,Y\in \mathcal S_1.$ The object $Z$ constructed by Algorithm {\rm \ref{algorithm}} is an~extension of $Y$ by $X.$
\end{lem}

\begin{proof}  For natural numbers $m,r,k$ such that $m\geq 1$ and $r>k$ let
$$
E_1(m,r,k):\;\;\;\; 0\rightarrow  P_0^{m} \rightarrow  P_0^{m+r-1}
\oplus P_1^{k+1} \rightarrow P_1^r\oplus P_0^k \rightarrow  0
$$
denotes the following short exact sequence:
$$\xymatrix{0\ar[r]&N_{(m)}\ar[rr]^-{ \left[\begin{smallmatrix}
 T^{r-1} \\
  -T^{k}
\end{smallmatrix}
\right]}&& N_{(m+r-1,k+1)}\ar[rr]^-{ \left[\begin{smallmatrix}
  1 & T^{r-k-1}\\
  0                                                                     & 1
\end{smallmatrix}
\right]} && N_{(r,k)} \ar[r] & 0\\
& 0\ar[u]\ar[rr]&& N_{(1)} \ar[u]^-{\left[\begin{smallmatrix}
 0 \\
 T^{k}
\end{smallmatrix}
\right]}
\ar[rr]^-{ \left[\begin{smallmatrix}
1
\end{smallmatrix}
\right]} && N_{(1)}\ar[u]^-{\left[\begin{smallmatrix}
T^{r-1}\\
   0
\end{smallmatrix}\right]}\ar[r]&0}$$
   and let
$$
E_2(m,r,k):\;\;\;\;  0\rightarrow  P_1^{m} \rightarrow  P_1^{m+r-1}
\oplus P_1^{k+1} \rightarrow P_1^r\oplus P_0^k \rightarrow  0
$$
for $m\geq 2$ denotes the following short exact sequence
$$\xymatrix{0\ar[r]&N_{(m)}\ar[rr]^-{ \left[\begin{smallmatrix}
 T^{r-1} \\ -T^{k}
\end{smallmatrix}
\right]}&& N_{(m+r-1,k+1)}\ar[rr]^-{ \left[\begin{smallmatrix}
  1 & T^{r-k-1}\\
  0&1
\end{smallmatrix}
\right]} && N_{(r,k)} \ar[r] & 0\\
0\ar[r] & N_{(1)}\ar[u]^-{\left[\begin{smallmatrix}
 T^{m-1}
\end{smallmatrix}
\right]}
\ar[rr]^-{\left[\begin{smallmatrix}
 1\\
  0
\end{smallmatrix}
\right]}
&& N_{(1,1)}\ar[u]^-{\left[\begin{smallmatrix}
 T^{m+r-2} &0\\
0&  T^{k}
\end{smallmatrix}
\right]}
\ar[rr]^-{ \left[\begin{smallmatrix}
0&
1
\end{smallmatrix}
\right]} && N_{(1)}\ar[u]^-{\left[\begin{smallmatrix}
T^{r-1}\\
   0
\end{smallmatrix}\right]}\ar[r]&0}$$
and for $m=1$ the following one
$$\xymatrix{0\ar[r]&N_{(1)}\ar[rr]^-{ \left[\begin{smallmatrix}
 0\\ T^{k}
\end{smallmatrix}
\right]}&& N_{(r,k+1)}\ar[rr]^-{ \left[\begin{smallmatrix}
  1 & 0\\
  0&1
\end{smallmatrix}
\right]} && N_{(r,k)} \ar[r] & 0\\
0\ar[r] & N_{(1)}\ar[u]^-{\left[\begin{smallmatrix}
 1
\end{smallmatrix}
\right]}
\ar[rr]^-{\left[\begin{smallmatrix}
 0\\
  1
\end{smallmatrix}
\right]}
&& N_{(1,1)}\ar[u]^-{\left[\begin{smallmatrix}
 T^{r-1} &0\\
0&  T^{k}
\end{smallmatrix}
\right]}
\ar[rr]^-{ \left[\begin{smallmatrix}
1&
0
\end{smallmatrix}
\right]} && N_{(1)}\ar[u]^-{\left[\begin{smallmatrix}
T^{r-1}\\
   0
\end{smallmatrix}\right]}\ar[r]&0}$$
We fix $X,Y\in \mathcal{S}_1$. Consider the following cases.

{\bf Case 1.} Assume that $\overline{\beta^X_1}=\overline{\beta^Y_1}$, i.e. the objects $X$ and $Y$ have the same number
of indecomposable direct summands. Note that in this case the Step 4 of the algorithm does not hold (because the loops start from
$\min\{\overline{\beta^X_1},\overline{\beta^Y_1}\}+1$). Moreover since $X,Y\in \mathcal{S}_1$, for any $i$ we have: if $\beta_i^*>\gamma_i^*$, then $\beta_i^*=\gamma_i^*+1$, where $*\in \{X,Y\}$.

It is easy to see that
in this case the  object $Z$ constructed by the algorithm is the~direct sum of the middle terms of the following
short exact sequences ($i=1,\ldots,\min \{\overline{\beta^X_1},\overline{\beta^Y_1}\}$):

$$
0\rightarrow P_0^{\beta_i^X}\rightarrow P_0^{\beta_i^X+\beta_i^Y}\rightarrow P_0^{\beta_i^Y}\rightarrow 0,
$$
if $\beta^X_i= \gamma^X_i$ and $\beta^Y_i= \gamma^Y_i$;

$$
0\rightarrow P_1^{\beta_i^X}\rightarrow P_1^{\beta_i^X+\beta_i^Y}\rightarrow P_0^{\beta_i^Y}\rightarrow 0,
$$
if $\beta^X_i> \gamma^X_i$ and $\beta^Y_i= \gamma^Y_i$;

$$
E_1(\beta_i^X, \beta_i^Y,0):\;\;\;\; 0\rightarrow P_0^{\beta_i^X}\rightarrow P_0^{\beta_i^X+\gamma_i^Y}\oplus P_1^1\rightarrow P_1^{\beta_i^Y}\rightarrow 0,
$$
if $\beta^X_i= \gamma^X_i$ and $\beta^Y_i> \gamma^Y_i$;

$$
E_2(\beta_i^X, \beta_i^Y,0):\;\;\;\; 0\rightarrow P_1^{\beta_i^X}\rightarrow P_1^{\beta_i^X+\gamma_i^Y}\oplus P_1^1\rightarrow P_1^{\beta_i^Y}\rightarrow 0,
$$
if $\beta^X_i> \gamma^X_i$ and $\beta^Y_i> \gamma^Y_i$.

{\bf Case 2.} Assume that $\overline{\beta^X_1}>\overline{\beta^Y_1}$, i.e. the object $X$ has more
indecomposable direct summands than  $Y$. It is easy to see that
in this case the object $Z$ constructed by the algorithm is the~direct sum of the middle terms of
the short exact sequences given in Case 1 for $i=1,\ldots,\min \{\overline{\beta^X_1},\overline{\beta^Y_1}\}$
and the short exact sequences (for $i=\min \{\overline{\beta^X_1},\overline{\beta^Y_1}\}+1,\ldots,\overline{\beta_1^X}$):
$$
   0\to P_1^{\beta^X_i}\to P_1^{\beta^X_i}\to 0\to 0.
$$
if $\beta_i>\gamma_i$;
$$
   0\to P_0^{\beta^X_i}\to P_0^{\beta^X_i}\to 0\to 0
$$
if $\beta_i=\gamma_i$.

{\bf Case 3.} Assume that $\overline{\beta^X_1}<\overline{\beta^Y_1}$, i.e. the object $X$ has less
indecomposable direct summands than  $Y$.


Let
$$I=\{i\in\{1,\ldots,\overline{\beta^X_1}\}\;\; ;\;\; \beta^Y_i>\gamma^Y_i\}=\{i_1,\ldots,i_n\}$$

$$J=\{j=\overline{\beta^X_1}+1,\ldots  ,\overline{\beta^Y_1}\;\; ;\;\; \beta^Y_j=\gamma^Y_j\}=\{j_1,\ldots,j_{|J|}\}$$

and we set the numeration as follows $j_i<j_{i+1}$ for all $i$.

Let $\overline{I}=\{i_1,\ldots,i_{\min\{n,|J|\}}\}$.

In this case the object $Z$ constructed by the algorithm is the~direct sum of the middle terms of the following
short exact sequences ($i=1,\ldots,\min \{\overline{\beta^X_1},\overline{\beta^Y_1}\}$):

$$
0\rightarrow P_0^{\beta_i^X}\rightarrow P_0^{\beta_i^X+\beta_i^Y}\rightarrow P_0^{\beta_i^Y}\rightarrow 0,
$$
if $\beta^X_i= \gamma^X_i$ and $\beta^Y_i= \gamma^Y_i$;

$$
0\rightarrow P_1^{\beta_i^X}\rightarrow P_1^{\beta_i^X+\beta_i^Y}\rightarrow P_0^{\beta_i^Y}\rightarrow 0,
$$
if $\beta^X_i> \gamma^X_i$ and $\beta^Y_i= \gamma^Y_i$;

$$
E_1(\beta_{i_s}^X, \beta_{i_s}^Y,\beta^Y_{j_s}):\;\;\;\; 0\rightarrow P_0^{\beta_i^X}\rightarrow P_0^{\beta_i^Z+\gamma_i^Y}\oplus P_1^{\beta^Y_{j_s}+1}\rightarrow P_1^{\beta_i^Y}\oplus P_0^{\beta^Y_{j_s}}\rightarrow 0,
$$
if $\beta^X_i= \gamma^X_i$, $\beta^Y_i> \gamma^Y_i$ and $i=i_s\in\ov{I}$;

$$
E_2(\beta_{i_s}^X, \beta_{i_s}^Y,\beta^Y_{j_s}):\;\;\;\; 0\rightarrow P_1^{\beta_i^X}\rightarrow P_1^{\beta_i^Z+\gamma_i^Y}\oplus P_1^{\beta^Y_{j_s}+1}\rightarrow P_1^{\beta_i^Y}\oplus P_0^{\beta^Y_{j_s}}\rightarrow 0,
$$
if $\beta^X_i> \gamma^X_i$, $\beta^Y_i> \gamma^Y_i$ and $i=i_s\in \ov{I}$.

The following two cases holds if the number $n$ is reduced to $0$ in the  Step 4 of Algorithm \ref{algorithm}:
$$
   0\to 0\to P_1^{\beta^Y_i}\to P_1^{\beta^Y_i}\to 0,
$$
if $\beta_i>\gamma_i$ and $i>\overline{\beta^X_1};$
$$
   0\to 0\to P_0^{\beta^Y_i}\to P_0^{\beta^Y_i}\to 0,
$$
if $\beta_i=\gamma_i$ and $i\in \{j_{\min\{n,|J|\}+1},\ldots,j_{|J|}\}. $

Otherwise the 5th Step of Algorithm is execute, the following two cases illustrated this situation:
$$
E_1(\beta_i^X, \beta_i^Y,0):\;\;\;\; 0\rightarrow P_0^{\beta_i^X}\rightarrow P_0^{\beta_i^Z+\gamma_i^Y}\oplus P_1^1\rightarrow P_1^{\beta_i^Y}\rightarrow 0,
$$
if $\beta^X_i= \gamma^X_i$, $\beta^Y_i> \gamma^Y_i$ and $i\in I\setminus \overline{I}$;

$$
E_2(\beta_i^X, \beta_i^Y,0):\;\;\;\; 0\rightarrow P_1^{\beta_i^X}\rightarrow P_1^{\beta_i^Z+\gamma_i^Y}\oplus P_1^1\rightarrow P_1^{\beta_i^Y}\rightarrow 0,
$$
if $\beta^X_i> \gamma^X_i$, $\beta^Y_i> \gamma^Y_i$ and $i\in I\setminus \overline{I}$.

The lemma is proved.

\end{proof}



The following two lemmata are used in the proof of Lemma \ref{lem_dommax}.

\begin{lem}\label{lem_rozNgamma}
Let $X=(N_{\alpha^X},N_{\beta^X},f_X),Y=(N_{\alpha^Y},N_{\beta^Y},f_Y)\in \mathcal S_1.$
If an~object $Z\in \mathcal S_1$ is an extension of $Y$ by $X$, then $Z$ is also an~extension
of $\widetilde{Y}$ by $\widetilde{X}$, where $\widetilde{X}=(N_{\alpha^X}\oplus N_{\alpha^Y},N_{\beta^X}\oplus N_{\alpha^Y},
f_X\times id_{ N_{\alpha^Y}})$ and $\widetilde{Y}=(0, N_{\gamma^Y},0)$ for $\gamma^Y$ such that $\Coker f_Y\simeq N_{\gamma^Y}.$
\end{lem}

\begin{proof}
If $Z$ is an~extension of $Y$ by $X$, then there exists a~commutative diagram with exact rows:
$$\xymatrix{0\ar[r] & N_{\beta^X}\ar[r]^{h_{XZ}} & N_{\beta^Z}\ar[r]^{h_{ZY}} & N_{\beta^Y}\ar[r] &0\\
0\ar[r] & N_{\alpha^X}\ar[r]_{g_{XZ}}\ar[u]^{f_X} & N_{\alpha^Z}\ar[r]_{g_{ZY}}\ar[u]^{f_Z} & N_{\alpha^Y}\ar[r]\ar[u]^{f_Y} &0}$$

Since the morphisms $f_X,f_X,f_Y$ are injective, this diagram induces the following commutative diagram with exact rows and columns.

$$
\xymatrix{ & 0 & 0 & 0&\\
0\ar[r] & N_{\gamma^X}\ar[r]_{j_{XZ}}\ar[u] & N_{\gamma^Z}\ar[r]_{j_{ZY}}\ar[u] & N_{\gamma^Y}\ar[r]\ar[u] &0\\
0\ar[r] & N_{\beta^X}\ar[r]^{h_{XZ}}\ar[u]^{f^{\prime}_X} & N_{\beta^Z}\ar[r]^{h_{ZY}}\ar[u]^{f^{\prime}_Z} & N_{\beta^Y}\ar[r]\ar[u]^{f^{\prime}_Y} &0\\
0\ar[r] & N_{\alpha^X}\ar[r]_{g_{XZ}}\ar[u]^{f_X} & N_{\alpha^Z}\ar[r]_{g_{ZY}}\ar[u]^{f_Z} & N_{\alpha^Y}\ar[r]\ar[u]^{f_Y} &0\\
& 0\ar[u] & 0\ar[u] & 0\ar[u]&}$$

Since we are working in the category $\mathcal{S}_1$ the bottom row splits
(because it contains only semisimple objects). Therefore there exists
$g'_{ZY}:N_{\alpha^Y}\to N_{\alpha^Z}$ such that $g_{ZY}\circ g'_{ZY}=\id_{N_{\alpha^Y}}$.
We consider the following diagram:
$$\xymatrix{0\ar[r]& N_{\beta^X}\oplus N_{\alpha^Y}\ar[rr]^-{[h_{XZ},f_Z\circ g'_{ZY}]} && N_{\beta^Z}\ar[r]^-{ f^{\prime}_Y\circ h_{ZY}} & N_{\gamma^Y}\ar[r]&0\\
0\ar[r]&N_{\alpha^X}\oplus N_{\alpha^Y}\ar[rr]_-{[g_{XZ},g'_{ZY}]}\ar[u]^{f_X\times id_{N_{\alpha^Y}}} && N_{\alpha^Z}\ar[r]\ar[u]^{f_Z} & 0\ar[u]\ar[r]&0}$$

It is straightforward to check that this diagram is commutative. Since $N_{\alpha^X}\oplus N_{\alpha^Y}$ and $N_{\alpha^Z}$
are semisimple nilpotent operators with the same dimension, the bottom row is exact.
Moreover $h_{ZY}\circ f_Z\circ g'_{ZY}=f_Y$ is injective. It follows that
$0=(f_Z\circ g'_{ZY}(N_{\alpha^X}))\cap \Ker h_{ZY}=(f_Z\circ g'_{ZY}(N_{\alpha^X}))\cap {\rm Im} h_{XZ}$. Now, it is easy to prove
that the morphism $[h_{XZ},f_Z\circ g'_{ZY}]$ is injective and the top row is exact. This finishes the proof.
\end{proof}

The following fact is proved in \cite[Lemma 3.3]{kos}.
\begin{lem}\label{lem_sum_part}
Let $\sigma, \nu,  \mu$ be partitions. If there exists a short exact sequence:
$$0\rightarrow N_{\nu}\rightarrow N_{\sigma}\rightarrow N_{\mu}\rightarrow 0,$$
then for any $m\geq 1:$
$$\sum\limits_{i=1}^m \sigma_i \leq \sum\limits_{i=1}^m \lambda_i,$$
where $\lambda = \nu+\mu.$
\end{lem}

In the following lemma we prove that the extension
computed by Algorithm \ref{algorithm} is minimal in the
dominance order $\domleq$.

\begin{lem}\label{lem_dommax}
If $\widetilde{Z}$ is the extension of $Y$ by $X$ constructed by Algorithm {\rm \ref{algorithm}},
then for any extension $Z$  of $Y$ by $X$ and for any $j\geq 1$ we have
$\sum\limits_{i=1}^j \gamma_j^{\widetilde{Z}} \geq  \sum\limits_{i=1}^j \gamma_j^{Z}$ and
$\sum\limits_{j=1}^i \beta_j^{\widetilde{Z}} \geq  \sum\limits_{j=1}^i \beta_j^{Z}$.
In particular $\widetilde{Z}\domleq Z$.

\end{lem}

\begin{proof}
Let $Z$ be an~arbitrary extension of $Y$ by $X$.
It  induces short exact sequences:
$$0\rightarrow N_{\beta^X}\rightarrow N_{\beta^{Z}} \rightarrow N_{\beta^Y}\rightarrow 0$$

$$0\rightarrow N_{\gamma^X}\rightarrow N_{\gamma^{Z}} \rightarrow N_{\gamma^Y}\rightarrow 0$$

Applying Algorithm \ref{algorithm} to $X$ and $Y$ we compute $\widetilde{Z}$.
Note that
$\gamma^{\widetilde{Z}}=\gamma^X+\gamma^Y$.
By Lemma \ref{lem_sum_part}, for all $i$, we get:
$$\sum\limits_{j=1}^i \gamma_j^Z \leq \sum\limits_{j=1}^i \gamma_j^{\widetilde{Z}}.$$

To prove the second part of lemma we consider $Z$
as an extension of $\widetilde{Y}=(0,N_{\gamma^Y},0)$ by
$\widetilde{X}=(N_{\alpha^X}\oplus N_{\alpha^Y},N_{\beta^X}\oplus N_{\alpha^Y},f_{\widetilde{X}})$.
It is possible by Lemma \ref{lem_rozNgamma}.

For all $i=1,\ldots, \min\{\overline{\beta_1^X},\overline{\beta_1^Y}\}$, by the Step 3 (a) of Algorithm \ref{algorithm},
we have
$\beta^{\widetilde{Z}}_i=\beta^X_i+\gamma^Y_i$. Since the number  of indecomposable direct summands of $N_{\beta^X}$
is equal to $\overline{\beta_1^X}$ and it is less than or equal to the number of indecomposable direct summands of
$N_{\beta^X}\oplus N_{\alpha^Y}$, we get
\begin{equation}
\beta^{\widetilde{Z}}_i=\beta^X_i+\gamma^Y_i=\beta^{\widetilde{X}}_i+\beta^{\widetilde{Y}}_i
\label{eq-tilde}
\end{equation}
for all $i=1,\ldots, \min\{\overline{\beta_1^X},\overline{\beta_1^Y}\}$.
By Lemma \ref{lem_sum_part} we get:
$$\sum\limits_{j=1}^i \beta_j^Z \leq \sum\limits_{j=1}^i (\beta_j^{\widetilde{X}}+\beta_j^{\widetilde{Y}})= \sum\limits_{j=1}^i \beta_j^{\widetilde{Z}},$$
for all $i=1,\ldots, \min\{\overline{\beta_1^X},\overline{\beta_1^Y}\}.$ If $\overline{\beta_1^X}=\overline{\beta_1^Y}$,
we are done.

Assume that $\overline{\beta_1^X}>\overline{\beta_1^Y}$ and $i>m=\min\{\overline{\beta_1^X},\overline{\beta_1^Y}\}$.
In this case $\beta_i^Y=\gamma_i^Y=0$.
By  Steps 4 and 5  of Algorithm \ref{algorithm}, we get
 $\beta_i^{\widetilde{Z}}=\beta_i^{\widetilde{X}}$. Lemma \ref{lem_sum_part} and the inequalities given just above imply:

 $$\sum\limits_{j=1}^i \beta_j^Z  \leq \sum\limits_{j=1}^i (\beta_j^{\widetilde{X}}+\beta_j^{\widetilde{Y}})\leq\sum\limits_{j=1}^m \beta_j^{\widetilde{Z}}+\sum\limits_{j=m+1}^i( \beta_j^{\widetilde{X}}+\beta_j^{\widetilde{Y}})=
 $$
$$
 =\sum\limits_{j=1}^m \beta_j^{\widetilde{Z}}+\sum\limits_{j=m+1}^i \beta_j^{\widetilde{X}}
 =\sum\limits_{j=1}^m \beta_j^{\widetilde{Z}}+\sum\limits_{j=m+1}^i \beta_j^{\widetilde{Z}}
 $$
for all $i$ and we are done.

Assume that $\overline{\beta_1^Y}>\overline{\beta_1^X}$. For  $i>m=\min\{\overline{\beta_1^X},\overline{\beta_1^Y}\}$,
by Step 4 of Algorithm \ref{algorithm},  we have
$$\beta^{\widetilde{Z}}_i=\beta^Y_i+\mathbf{1}{\{\gamma_Y^i=\beta_Y^i\,{\rm and}\,n>0\}}$$
  Consider two cases:
 \begin{itemize}
 \item Assume that $n>0$ for $i$. Note that $n>0$ for all $m\leq j<i$.
  Steps 4 and 5 of Algorithm \ref{algorithm} gives
 $$ \beta^{\widetilde{Z}}_i=\gamma_i^Y+1,
 $$
  because for $\gamma^Y_i\neq\beta^Y_i$ we have $\beta_i^Y=\gamma_i^Y+1$ (we are working in category $\mathcal{S}_1$) and otherwise $\mathbf{1}{\{\gamma_Y^i=\beta_Y^i\,{\rm and}\,n>0\}}=1$.  By the definition of $\widetilde{X}$ and
  $\widetilde{Y}$ we get
 $$
 \gamma_i^Y+1=\beta_i^{\widetilde{Y}}+\beta_i^{\widetilde{X}}$$
 and by \ref{eq-tilde} we get $\beta_j^{\widetilde{Z}}=\beta_j^{\widetilde{Y}}+\beta_j^{\widetilde{X}}$ for all $j\leq i$.
  By Lemma \ref{lem_sum_part} applied  to partitions $\beta^{\widetilde{Y}},\beta^{\widetilde{X}},\beta^Z$  we obtain:
$$\sum\limits_{j=1}^i \beta_j^Z\leq
\sum\limits_{j=1}^i (\beta_j^{\widetilde{X}}+\beta_j^{\widetilde{Y}}) =\sum\limits_{j=1}^i \beta_j^{\widetilde{Z}}
.
$$
 \item Assume that $n=0$ for $i$.  Applying Steps 3 and 4 of Algorithm \ref{algorithm}, it is easy
 to deduce that
 $$\sum\limits_{j=1}^i \beta_j^{\widetilde{Z}}=\sum\limits_{j=1}^i \beta_j^{X}+\sum\limits_{j=1}^{i} \beta_j^{Y}.$$
By Lemma \ref{lem_sum_part} applied to partitions $\beta^{Y},\beta^{X}$ and $\beta^Z$ we obtain:
$$\sum\limits_{j=1}^i \beta_j^Z \leq \sum\limits_{j=1}^i \beta_j^{X}+\sum\limits_{j=1}^{i} \beta_j^{Y}= \sum\limits_{j=1}^i \beta_j^{\widetilde{Z}}.$$
\end{itemize}
\end{proof}

\begin{cor}
 \label{cor_gen}
 Let $K$ be an algebraically closed field. The extension $Y\ast X$ computed by Algorithm \ref{algorithm} is the generic extension of $Y$ by $X$.
\end{cor}

\begin{proof}
 By Lemma \ref{lem_dommax} any extension $Z$ of $Y$ by $X$ satisfies
 $Y\ast X\domleq Z$ and hence 
 $Y\ast X\degleq Z$, see Theorem \ref{thm-main}. It follows that
 $\mathcal{O}_Z\subseteq\overline{\mathcal{O}}_{Y\ast X}$.
Therefore any extension $Z$ of $Y$ by $X$
that it not isomorphic to $Y\ast X$ satisfies $\dim \mathcal{O}_Z< \dim \mathcal{O}_{Y\ast X}$ and 
$\dim_K\End_\mathcal{S}(Z)>\dim_K\End_\mathcal{S}(Y\ast X)$ (by the formula for the dimension of an~orbit). We are done.
\end{proof}

\section{Properties of generic extensions}\label{sub-gen}

In Section \ref{sec-monoid} we prove that the operation $\ast$ of taking the generic extension provides
the set of isomorphism classes of objects in $\mathcal{S}_1$ in an~associative monoid structure (Lemma \ref{lem_lacznosc}) and
we describe generators of this monoid (Lemma \ref{lem-gen-monoid}). For this we need Lemma \ref{lem_deg} (which is a~key lemma in this section and it is also used in the proof of Theorem \ref{thm-main1}). We start with a~technical fact
that one can easily deduce analyzing 
Algorithm \ref{algorithm}.

\begin{lem}
Let $X,Y\in \mathcal S_1$  be defined by $(\gamma^X,\beta^X)$ and $(\gamma^Y,\beta^Y)$, respectively. Algorithm \ref{algorithm} constructs the object $Z=Y*X$ having  the following properties
\begin{enumerate}
\item [\rm (1)]  $\beta^Z_i=\beta^X_i+\gamma^Y_i$, for $i=1,\ldots,\overline{\beta_1^X}$,	

\item [\rm (2)]
$
\sum_{i=1}^k\beta_i^{Y*X}=\sum_{i=1}^k (\gamma_i^Y+\beta_i^X)+\min\{k-\overline{\beta_1^X},\sum_{i=1}^k(\beta_i^Y-\gamma_i^Y)\},
$
for all $k>\overline{\beta_1^X}$.
\end{enumerate}
\label{lem-alg}
\end{lem}

\begin{proof}
The statement (1) follows from the Steps 3 and 4 of Algorithm \ref{algorithm}	and the fact that $\beta^Y_i=\gamma^Y_i=0$ for $i=\min\{\overline{\beta^X_1},\overline{\beta^Y_1}\}+1,\ldots,\overline{\beta^X_1}$.

The statement (2) we prove inductively. We have only consider the case $\overline{\beta^Y_1}>\overline{\beta^X_1}$. Let
$k=\overline{\beta^X_1}+1$. We have $\beta_k^X=0$ and
$$
\sum_{i=1}^k\beta_i^{Y*X}=\sum_{i=1}^{\overline{\beta^X_1}} (\gamma_i^Y+\beta_i^X)+\beta_{k}^Y+
\mathbf{1}{\{\gamma_{k}^Y=\beta_{k}^Y\,
	{\rm and}\,n>0\}}
$$
Note that $$\min\{k-\overline{\beta^X_1},\sum_{i=1}^k(\beta_i^Y-\gamma_i^Y)\}=\min\{1,n+\beta_{k}^Y-\gamma_{k}^Y\}
=\left\{
\begin{array}{ll}
1&\mbox{ if } n>0 \mbox{ or } \beta_k^Y>\gamma_k^Y \\
0&\mbox{ otherwise}
\end{array}
\right.
$$
It follows that
$\gamma_k^Y+\beta_k^X+\min\{1,n+\beta_{k}^Y-\gamma_{k}^Y\}=\beta_{k}^Y+
\mathbf{1}{\{\gamma_{k}^Y=\beta_{k}^Y\,
	{\rm and}\,n>0\}}.
$
In the induction step we apply similar arguments.
\end{proof}

\begin{lem}\label{lem_deg}
Let $Y,Y'\in\mathcal{S}^b_a$ for some natural numbers $a,b$ and let $Y\degleq Y'$. For any $X\in \mathcal{S}_1$ following condition holds:
\begin{enumerate}
\item $Y*X\degleq Y'*X,$
\item $X*Y\degleq X*Y'.$
\end{enumerate}
\end{lem}
\begin{proof}
Let $Y,Y'\in\mathcal{S}^b_a$ and $Y\degleq Y'$.
By Theorem \ref{thm-main} the inequality $Y\domleq Y'$ holds. Therefore for any natural $k$ we have
\begin{equation}
\sum_{i=1}^k \lambda_i^Y\geq \sum_{i=1}^k \lambda_i^{Y'},
\label{ineqlam}
\end{equation}
where $\lambda\in\{\beta,\gamma\}.$
By Step 1 of Algorithm \ref{algorithm} we have
$\gamma^{Y* X}=\gamma^X+\gamma^Y$. Therefore
$$\sum_{i=1}^k\gamma_i^{Y*X}=\sum_{i=1}^k (\gamma_i^Y+\gamma_i^X)\geq\sum_{i=1}^k (\gamma_i^{Y'}+\gamma_i^X)=\sum_{i=1}^k \gamma_i^{Y'*X}$$
for any $k$.
It follows that $\gamma^{Y*X}\natleq \gamma^{Y'*X}$ and similarly $\gamma^{X*Y}\natleq \gamma^{X*Y'}$.

By Lemma \ref{lem-alg} and the inequality \ref{ineqlam}, for $k\leq\overline{\beta_1^X}$ we have
$$\sum_{i=1}^k\beta_i^{Y*X}=\sum_{i=1}^k (\gamma_i^Y+\beta_i^X)
\geq\sum_{i=1}^k (\gamma_i^{Y'}+\beta_i^X)=\sum_{i=1}^k \beta_i^{Y'*X}$$
and for $k>\overline{\beta_1^X}$
\begin{displaymath}
\sum_{i=1}^k\beta_i^{Y*X}=\sum_{i=1}^k (\gamma_i^Y+\beta_i^X)+\min\{k-\overline{\beta}_1^X,\sum_{i=1}^k(\beta_i^Y-\gamma_i^Y)\}.
\end{displaymath}
 If $\min\{k-\overline{\beta_1^X},\sum_{i=1}^k(\beta_i^Y-\gamma_i^Y)\}=k-\overline{\beta}_1^X$ then
\begin{multline*}
\sum_{i=1}^k\beta_i^{Y*X}=\sum_{i=1}^k (\gamma_i^Y+\beta_i^X)+k-\overline{\beta}_1^X\geq\sum_{i=1}^k (\gamma_i^{Y'}+\beta_i^X)+k-\overline{\beta}_1^X\geq \\
\sum_{i=1}^k (\gamma_i^{Y'}+\beta_i^X)+\min\{k-\overline{\beta}_1^X,\sum_{i=1}^k(\beta_i^{Y'}-\gamma_i^{Y'})\}=\sum_{i=1}^k\beta_i^{Y'*X},
\end{multline*}
otherwise
\begin{multline*}
\sum_{i=1}^k\beta_i^{Y*X}=\sum_{i=1}^k (\gamma_i^Y+\beta_i^X)+\sum_{i=1}^k(\beta_i^Y-\gamma_i^Y)=\sum_{i=1}^k(\beta_i^X+\beta_i^Y)\geq\\
\sum_{i=1}^k(\beta_i^X+\beta_i^{Y'})=\sum_{i=1}^k (\gamma_i^{Y'}+\beta_i^X)+\sum_{i=1}^k(\beta_i^{Y'}-\gamma_i^{Y'})\geq \\
\sum_{i=1}^k (\gamma_i^{Y'}+\beta_i^X)+\min\{k-\overline{\beta}_1^X,\sum_{i=1}^k(\beta_i^{Y'}-\gamma_i^{Y'})\}=\sum_{i=1}^k\beta_i^{Y'*X}.
\end{multline*}
Therefore $\beta^{Y*X}\natleq \beta^{Y'*X}$ and
 $Y*X\domleq Y'*X$. Again by Theorem \ref{thm-main} we obtain $Y*X\degleq Y'*X$.

Since $\overline{\beta_1^Y}\leq\overline{\beta_1^{Y'}}$ (or equivalently $\beta_1^Y\geq\beta_1^{Y'}$), in the same way we can see that $X*Y\domleq X*Y'$ for any $X\in \mathcal{S}.$
\end{proof}

\subsection{The monoid of generic extensions}\label{sec-monoid}

The associative property of $*$ is established by the next lemma.

\begin{lem}\label{lem_lacznosc}
If $X,Y,Z\in\mathcal S_1$, then $X*(Y*Z)=(X*Y)*Z.$
\end{lem}
\begin{proof}
We follow the arguments given in
\cite{reineke}. Consider the following diagram
$$\begin{array}{ccccccccc} &&&&0&&0&&\\
&&&&\downarrow&&\downarrow&&\\ 0&\to&Z&\to&T&\to&Y&\to&0\\
&&\parallel&&\downarrow&&\downarrow&&\\ 0&\to&Z&\to&(X* Y)*
Z&\to&X* Y&\to&0
\\ &&&&\downarrow&&\downarrow&& \\ &&&&X&=&X&& \\ &&&&\downarrow&&\downarrow&& \\
&&&&0&&0&&
\end{array}$$ where the top row is a~pull-back of the bottom. It follows from Lemma \ref{lem_dommax} that $Y*Z\degleq T$ and $X* T \degleq (X* Y)* Z$.
Therefore by   Lemma \ref{lem_deg} we have $X* (Y* Z) \degleq (X* Y)* Z$.

Dually, applying push out we get $(X* Y)* Z\degleq X* (Y * Z)$. This finishes the proof (because $\degleq$ is a~partial order).\end{proof}

\begin{rem}
 We can consider the associative monoid
 $$\mathcal{M}=(\mathcal{M},[0],*),$$
 where $\mathcal{M}$ is the set of isomorphism classes
 of objects in $\mathcal{S}_1$, $[0]$ is the isomorphism class of the zero object and $*$ is the operation of taking the generic extensions: $[X]* [Y]=[X* Y]$.
\end{rem}

\begin{lem}
 The monoid $\mathcal{M}$ is generated by the set
 $$
 \{[P_1^1],[(P_0^1)^{\oplus n}] \mbox{ for } n\geq 1\}.
 $$
\label{lem-gen-monoid}\end{lem}

\begin{proof}
 Note that
 \begin{itemize}
  \item $[P_0^m]=[P_0^1]^{* m}$ for all $m\geq 1$,
  \item $[P_1^m]=[P_0^1]^{* (m-1)}* [P_1^1]$ for all $m\geq 1$,
  \item $[(P_1^1)^{\oplus m}]=[P_1^1]^{* m}$ for all $m\geq 1$,
  \item $[(P_0^1)^{\oplus n}\oplus (P_1^1)^{\oplus m}]=
    [P_1^1]^{* m}* [P_0^1]^{\oplus n}$  for all $m,n\geq 1$.
 \end{itemize}
   Finally, let $m_1,\ldots,m_s\geq 2$, $m_{s+1}=\ldots =m_t=1$ and
   $\varepsilon_i\in\{0,1\}$ for all $i=1,\ldots,s+t$. We get
    $$
   [\bigoplus_{i=1}^{s+t}P_{\varepsilon_i}^{m_i}]=[(P_0^1)^{\oplus s}\oplus \bigoplus_{i=s+1}^{t}P_{\varepsilon_i}^1]*[\bigoplus_{i=1}^{s}P_{\varepsilon_i}^{m_i-1}].
   $$
   By induction we are done.
\end{proof}

\section{Generic extensions - geometry}
\label{section-extentions-geom}

In this section we present geometric interpretation of generic extensions and we prove Theorem \ref{thm-main1}.

Let $K$ be an~algebraically closed field. Fix subsets $\mathcal{Y}\subseteq \mrep(K)$, $\mathcal{X}\subseteq {}^{c}V^d(K)$ and $Y=(f,\varphi)\in \mathcal{Y}$,
$X=(g,\phi)\in \mathcal{X}$.  Let
$$
Z(Y,X)=\left\{(h,\nu)\in \mathbb{M}_{d\times a}\times \mathbb{M}_{d\times b}\;\; ;\;\;
\left(\left[\begin{array}{cc}g&h\\ 0&f\end{array}\right],
\left[\begin{array}{cc}\phi&\nu \\ 0&\varphi\end{array}\right]\right)
\in {}^{a+c}V^{b+d}(K)
\right\}
$$
and
\begin{multline*}
\mathcal{Z}(\mathcal{Y},\mathcal{X})=
\left\{
\left(
\left[\begin{array}{cc}g&h\\ 0&f\end{array}\right],
\left[\begin{array}{cc}\phi&\nu \\ 0&\varphi\end{array}\right]
\right)
;\right.\\ \left.\left.\left.
Y=(f,\varphi)\in \mathcal{Y},
X=(g,\phi)\in \mathcal{X}, (h,\nu)\in Z(Y,X)
\right.\right.\right\}.
\end{multline*}

\begin{lem}
If the sets $\mathcal{Y}\subseteq \mrep$, $\mathcal{X}\subseteq {}^{c}V^d$ are closed, then  
the set $\mathcal{Z}(\mathcal{Y},\mathcal{X})\subseteq {}^{a+c}V^{b+d}$ is closed.
 \label{lem-z(x.y)-closed}\end{lem}

\begin{proof}
Since the sets $\mathcal{Y}\subseteq \mrep$, $\mathcal{X}\subseteq {}^{c}V^d$ are closed (in the Zariski topology), there exists a~finite
set of equations deciding whether
$Y=(f,\varphi)\in \mathcal{Y},
X=(g,\phi)\in \mathcal{X}$. If $f,g$
have maximal rank, then $\left[\begin{array}{cc}g&h\\ 0&f\end{array}\right]$ has also maximal rank. Now, it is easy to see that there exists a~finite set of equations deciding whether an element belongs  to the set $\mathcal{Z}(\mathcal{Y},\mathcal{X})$. It follows that
the set $\mathcal{Z}(\mathcal{Y},\mathcal{X})\subseteq {}^{a+c}V^{b+d}$ is closed.
\end{proof}

Let
$$\mathcal{E}(\mathcal{Y},\mathcal{X})=\Gl(a+c,b+d)\cdot \mathcal{Z}(\mathcal{Y},\mathcal{X})\subseteq {}^{a+c}V^{b+d}.$$
Note that
$
\mathcal{E}(\mathcal{Y},\mathcal{X})
$
is the set of all  extensions $Z$  of some $Y\in \mathcal{Y}$ by some $X\in \mathcal{X}$, i.e.
there exists a~short exact sequence
$$
0\to X\to Z\to Y\to 0.
$$

\begin{lem}
Let $Y\in {}^{a}V^{b}$, $X\in {}^{c}V^{d}$.
The set $\mathcal{E}(\overline{\mathcal{O}_Y},\overline{\mathcal{O}_X})$ is closed in ${}^{a+c}V^{b+d}$.
 \label{lem-e(x.y)-closed}
\end{lem}

\begin{proof}
The set $\mathcal{Z}(\overline{\mathcal{O}_Y},\overline{\mathcal{O}_X})$
is stable under the action of the parabolic subgroup
$$\left[
\begin{array}{cc}
\Gl(a)&*\\
0& \Gl(c)
\end{array}
\right]
\times
\left[
\begin{array}{cc}
\Gl(b)&*\\
0& \Gl(d)
\end{array}
\right]
$$
of the group $\Gl(a+c,b+d).$
Moreover, by Lemma \ref{lem-z(x.y)-closed}, the set $\mathcal{Z}(\overline{\mathcal{O}_Y},\overline{\mathcal{O}_X})$
is closed. Therefore
by \cite[Proposition 6.6]{kraft1} the set
$\mathcal{E}(\overline{\mathcal{O}_Y},\overline{\mathcal{O}_X})=
\Gl(a+c,b+d)\cdot \mathcal{Z}(\overline{\mathcal{O}_Y},\overline{\mathcal{O}_X})$ is closed.
\end{proof}


\begin{proof}[Proof of Theorem \ref{thm-main1}] Let $Y\in \mrep$, $X\in {}^{c}V^d$.

(1) and (2) By the classification of objects in $\mathcal{S}_1$, recalled in Section \ref{section-pickets}, it follows that 
$\mathcal{E}(\mathcal{O}_Y,\mathcal{O}_X)$
is a~union of finitely many orbits $\mathcal{O}_Z$ corresponding to extensions 
$Z$ of $Y$ by $X$. By Lemmma \ref{lem_dommax} and Theorem \ref{thm-main} the object $Y\ast X$ computed
by Algorithm \ref{algorithm} 
is
the unique extension of $Y$ by $X$
that is minimal in the degeneration order.
Therefore $\mathcal{O}_Z\subseteq \overline{\mathcal{O}_{Y* X}}$ for any
extension $Z$ of $Y$ by $X$. It follows that $\mathcal{O}_{Y* X}$ is the unique dense orbit in $\mathcal{E}(\mathcal{O}_Y,\mathcal{O}_X)$. 
As an~immediate consequence we get
$\overline{\mathcal{O}_{Y* X}}=
\overline{\mathcal{E}(\mathcal{O}_Y,\mathcal{O}_X)}$.

(3)
We prove that $\mathcal{E}(\overline{\mathcal{O}_Y},\overline{\mathcal{O}_X})
\subseteq \overline{\mathcal{O}_{Y* X}}$.
Let $Z$ be an~extension of $Y'$ by $X'$ such that $Y'\in \overline{\mathcal{O}_Y}$ and $X'\in \overline{\mathcal{O}_X}$.
It means that $Y\degleq Y'$
and $X\degleq X'$. By Lemma \ref{lem_deg} and (1), we have
$Y* X\degleq Y'* X'\degleq Z$. Therefore $Z\in \overline{\mathcal{O}_{Y* X}}$.

Consequently, we have $\mathcal{E}(\mathcal{O}_Y,\mathcal{O}_X)\subseteq \mathcal{E}(\overline{\mathcal{O}_Y},\overline{\mathcal{O}_X})
\subseteq \overline{\mathcal{O}_{Y* X}}$.
Since the set $\mathcal{E}(\overline{\mathcal{O}_Y},\overline{\mathcal{O}_X})$ is closed and $\overline{\mathcal{O}_{Y* X}}=
\overline{\mathcal{E}(\mathcal{O}_Y,\mathcal{O}_X)}$,
we get  $\overline{\mathcal{E}(\mathcal{O}_Y,\mathcal{O}_X)} = \mathcal{E}(\overline{\mathcal{O}_Y},\overline{\mathcal{O}_X})
 = \overline{\mathcal{O}_{Y* X}}$.


(4) Since the group $\Gl(a,b)$ is irreducible for all $a,b$,  the sets $\mathcal{O}_X$,
$\overline{\mathcal{O}_X}$ are irreducible
for all $X$. By (3) we are done.
\end{proof}

%
%

%






\end{document}